\documentclass{amsart}
\usepackage{amssymb,amsmath,latexsym}
\usepackage{amsthm}
\usepackage{fontenc}
\usepackage{amssymb}
\numberwithin{equation}{section}

\newtheorem{theorem}{Theorem}[section]

\begin{document}
\author{Alexander E Patkowski}
\title{More on some Mock theta Double sums}

\maketitle
\begin{abstract} We offer some further applications of some Bailey pairs related to some mock theta functions which were established in a recent study. We discuss and offer some double-sum $q$-series, with new relationships among mock theta functions. We also offer a new relationship between the Bailey pair of Bringmann and Kane with that of Andrews.  \end{abstract}



\section{Introduction}
Recently [14] we offered some new expansions involving indefinite quadratic forms for mock theta functions by establishing new Bailey pairs. This method had its beginnings with Andrews [1,2], and subsequent work
by Zwegers established the connection with real analytic modular forms [17]. Mortensen and Hickerson established a comprehensive study of double Hecke-type sums for mock theta functions [10], which allows one to 
establish precisely when a Hecke-type expansion ([8]) is a mock modular form. Lovejoy and Osburn have recently given examples of double-sum mock theta functions, and connections to known single-sum mock theta functions [13]. (See also the important related work on mock theta functions [6].) For standard $q$-series notation and background information on Bailey's lemma see [3, 16], and recall that $(x;q)_n=(x)_n:=(1-x)(1-xq)\cdots(1-xq^{n-1}).$ Throughout we put $q\in\mathbb{C},$ and $0<|q|<1.$
\par The paper is organized as follows. First we establish the main Bailey pairs which give our motivation for further comments on Mock theta functions, adding to the applications offered in [1]. In the following section, we establish some interesting new double sums and further considerations. Our proof of our main pair implies an alternative proof to the one given by Lovejoy and Osburn [14], that the pair of Andrews [1, Lemma 12] may be obtained by the pair of Bringmann and Kane [5, Theorem 2.3].

\section{Applications and Further considerations}
We say denote $(\alpha_n(a,q), \beta_n(a,q))$ to be a Bailey pair where
\begin{equation}\beta_n(a,q)=\sum_{0\le j\le n}\frac{\alpha_n(a,q)}{(q;q)_{n-j}(aq;q)_{n+j}}.\end{equation}
If $q$ is not raised to a power then we just write $(\alpha_n(a,q), \beta_n(a,q))=(\alpha_n, \beta_n).$ 
We write down [15, (S1)]: If $(\alpha_n, \beta_n)$ is a Bailey pair relative to $a$ then so is $(\alpha_n', \beta_n')$ where

\begin{equation}\alpha_n'=a^{n}q^{n^2}\alpha_n,\end{equation}
\begin{equation}\beta_n'=\sum_{0\le j\le n}\frac{a^{j}q^{j^2}}{(q)_{n-j}}\beta_j.\end{equation}
Lastly, we write down [15, (E1)]: If $(\alpha_n(a^2,q^2), \beta_n(a^2,q^2))$ is a Bailey pair relative to $a^2$ then so is $(\alpha_n'(a^4,q^4), \beta_n'(a^4,q^4))$ where
\begin{equation}\alpha_n'(a^4,q^4)=\alpha_n(a^2,q^2),\end{equation}
\begin{equation}\beta_n'(a^4,q^4)=\sum_{j\ge0}^{n}\frac{(-1)^{n-j}q^{2(n-j)^2}}{(-a^2q^2;q^2)_{2n}(q^4;q^4)_{n-j}}\beta_{j}(a^2,q^2).\end{equation}

\begin{theorem} The pair of sequences $(\dot{\alpha}_n(q,q), \dot{\beta}_n(q,q))$ form a Bailey pair where
\begin{equation}\dot{\alpha}_n(q,q)=q^{-n^2-n}(a_n(q^{1/2})+a_n(-q^{1/2})),\end{equation}
\begin{equation}\dot{\beta}_n(q,q)=\frac{2(-1)^n}{(q^2;q^2)_n(1-q^{2n+1})},\end{equation}
where,
$$a_n(q)=q^{3n^2+2n}\frac{(1-q^{2n+1})}{1-q^2}\sum_{|j|\le n}(-1)^jq^{-j^2}.$$

\end{theorem}
\begin{proof} We recall the recently found Bailey pair [13] $(\alpha_n, \beta_n)$ relative to $(q^4, q^4)$ where
\begin{equation}\beta_n(q^4,q^4)=\frac{2}{(-q^4;q^2)_{2n}(q^2;q^4)_{n+1}},\end{equation}
\begin{equation}\alpha_n(q^4,q^4)=q^{n(n+1)}\left(\frac{(-1)^nq^{2n^2+n}(1+q^{2n+1})}{(1-q^2)}\sum_{|j|\le n}q^{-j^2}+\frac{q^{2n^2+n}(1-q^{2n+1})}{(1-q^2)}\sum_{|j|\le n}(-1)^jq^{-j^2}\right).\end{equation}
If we choose (2.8)--(2.9) to be the left side of (E1), (2.4)--(2.5), then by the uniqueness of Bailey pairs we see that we would need the Bailey pair
\begin{equation}\bar{\alpha}_n(q^2)=(a_n(q)+a_n(-q)),\end{equation}
\begin{equation}\bar{\beta}_n(q^2)=2\sum_{j\ge0}^{n}\frac{q^{2(n-j)}}{(q^4;q^4)_{n-j}(q^2;q^4)_{j+1}},\end{equation}
on the right side. Put the $q^2\rightarrow q$ version of this pair on the left side of the $a=q$ case of (S1), (2.2)--(2.3), to get that this implies the Bailey pair 
in the theorem by the uniqueness of Bailey pairs. This follows from observing that $\bar{\beta}_n(q)$ is the coefficient of $z^n$ in 
$$\frac{1}{(zq;q^2)_{\infty}}\sum_{n\ge0}\frac{z^n}{(q;q^2)_{n+1}},$$
which is the coefficient of $z^n$ in (by [14, eq.(2.12)])
$$\frac{1}{(z)_{\infty}}\sum_{n\ge0}\frac{z^n(-1)^nq^{n(n+1)}}{(q^2;q^2)_n(1-q^{2n+1})},$$
which is 
$$\sum_{j\ge0}^{n}\frac{(-1)^{j}q^{j(j+1)}}{(q)_{n-j}(q^2;q^2)_{j}(1-q^{2j+1})}.$$
\end{proof}
Note: We may now see that the reverse process implies a proof of the Bailey pair of Andrews [1, Lemma 12] by using the pair of Bringmann and Kane [5, Theorem 2.3].
\\*
We now offer some special cases, which include new mock modular form double sums and a modular form as examples. Both mock theta functions may be found in [1, 3]. For this we need a simple form of Bailey's lemma [4]
\begin{equation}\sum_{n\ge0}(X)_n(Y)_n(aq/XY)^n\beta_n=\frac{(aq/X)_{\infty}(aq/Y)_{\infty}}{(aq)_{\infty}(aq/XY)_{\infty}}\sum_{n\ge0}\frac{(X)_n(Y)_n(aq/XY)^n\alpha_n}{(aq/X)_n(aq/Y)_n}.\end{equation}
\begin{theorem} We have that, 
\begin{equation} 2\sum_{n\ge0}\sum_{n\ge j\ge 0}\frac{(-1)^j q^{n^2+n+j(j+1)/2}}{(-q)_n(q)_{n-j}(q)_{j}(1-q^{2j+1})}=A(q^{1/2})+A(-q^{1/2}),\end{equation}
where 
$$A(q)=\sum_{n\ge0}^{\infty}\frac{q^{2n^2+2n}}{(-q)_{2n+1}},$$ and
\begin{equation} \sum_{n\ge0}\sum_{n\ge j\ge 0}\frac{(-1)^j q^{2n^2+2n+j^2+j}}{(-q)_{2n+1}(q^2;q^2)_{n-j}(q^2;q^2)_{j}(1-q^{2j+1})}=F_2(q^2),\end{equation}
where
$$F_2(q)=\sum_{n\ge0}\frac{q^{n^2+n}}{(q^{n+1})_{n+1}}.$$ 
\begin{equation} \sum_{n\ge0}\sum_{n\ge j\ge 0}\frac{(-1)^n q^{2n^2+2n+(n-j)^2}}{(-q)_{2n+1}(q^2;q^2)_{n-j}(q^2;q^2)_{j}(1-q^{2j+1})}=\frac{\phi(q^2)}{(-q^2;q^2)_{\infty}},  \end{equation}
where  $$\phi(q)=\sum_{n\ge0}\frac{q^{n(n+1)/2}}{(q;q^2)_{n+1}}.$$
And further,
\begin{equation} 2\sum_{n\ge0}\sum_{n\ge j\ge 0}\frac{(-1)^j q^{n(n+1)/2+j(j+1)/2}}{(q)_{n-j}(q)_{j}(1-q^{2j+1})}=\frac{(-q)_{\infty}}{(q)_{\infty}}((q)_{\infty}(q^{1/2};q^{1/2})_{\infty}+(q)_{\infty}^2(-q^{1/2};q)_{\infty}). \end{equation}
\end{theorem}
\begin{proof} For (2.13) we use the $X, Y \rightarrow\infty$ case of (2.12) with the Bailey pair in Theorem 2.1 and [15, (S2)], then compare with [3, eq.(1.15)]. For (2.14) we use the $X, Y \rightarrow\infty$ case of (2.12) with the Bailey pair in Theorem 2.1 and [15, (E2)], then compare with [14, Theorem 2]. For (2.15) we use the $X, Y \rightarrow\infty$ case of (2.12) with the Bailey pair in Theorem 2.1 and [15, (E1)], then compare with [14, Theorem 1]. For (2.16) we use the $X=-q,$ $Y \rightarrow\infty,$ case of (2.12) with the Bailey pair in Theorem 2.1 and [15, (S2)], and then apply the well-known expansion of Kac and Peterson [9] $\sum_{n\ge0}q^{2n^2+n}(1-q^{2n+1})\sum_{|j|\le n}(-1)^jq^{-j^2}=(q)_{\infty}(q^2;q^2)_{\infty}.$
\end{proof}
We mention that (2.13) is a mock theta function, and bears such a close resemblance to the double sums included in Lovejoy and Osburn's list of Mock theta functions [13], that it is worth considering small variations of their functions as well. [15, (E2)] may similarly be applied to our Bailey pair using the $X=-q,$ $Y \rightarrow\infty,$ case to obtain a double sum for $\phi(q^2).$ \par
Now we consider the Hecke form
\begin{equation} \sum_{n\ge0}q^{(k+1)n^2+kn}(1-q^{2n+1})\sum_{|j|\le n}(-1)^j q^{-j^2}, \end{equation}
which has special cases $k=2, 3, 4, 6$ as known mock theta functions when multiplied by an appropriate modular form [3, 14].
Lovejoy and Osburn have noted that standard application of the Bailey chain does not imply a full general mock modular form [11]. Recall the double sum 
\begin{equation} f_{a,b,c}(x,y,q):=\left(\sum_{r,s\ge0}-\sum_{r,s<0}\right)(-1)^{r+s}x^{r}y^{s}q^{ar(r-1)/2+brs+cs(s-1)/2},\end{equation}
which is key to establishing connections to mock theta functions through the work done in [10].
We state a general result which we believe might be worth looking into further, as it contains (2.17) and mock theta expansions found in [3, 14]. Note that this result is also a corollary of a Bailey pair from Lovejoy's work [12, Theorem 1.2, $l=0,$ $k=1$].

\begin{theorem} \it We have for integers $k>1,$ 
$$\sum_{n_1\ge0}\cdots\sum_{n_{k}\ge0}\frac{q^{n_1^2+n_1+n_2^2+n_2+\cdots+n_k^2+n_k}(-1)^{n_k}}{(q)_{n_1-n_2}(q)_{n_2-n_3}\cdots(q)_{n_{k-1}-n_{k}}(q^2;q^2)_{n_{k}}}$$
\begin{equation}=\frac{1}{(q)_{\infty}}(f_{k,k+2,k}(q^{2k},q^{2k},q^2)+q^{2k+1}f_{k,k+2,k}(q^{4k+2},q^{4k+2},q^2)).\end{equation}
Further, for $k=2,3,4,6$ the multi-sum on the left side is a mixed mock modular form.
\rm
\end{theorem}
\begin{proof} We take the Bailey pair relative to $q$ from the paper [10], where
$$\alpha_n=\frac{q^{n^2}(1-q^{2n+1})}{1-q}\sum_{|j|\le n}(-1)^j q^{-j^2},$$
$$\beta_{n}=\frac{(-1)^n}{(q^2;q^2)_n},$$

 and insert it into the Bailey chain [4], and then note that (2.17) may be written in the form (2.18) as
$$\sum_{n\ge0}\sum_{|j|\le n}(-1)^jq^{(k+1)n^2+kn^2-j^2}-\sum_{n<0}\sum_{|j|\le -n-1}(-1)^jq^{(k+1)n^2+kn^2-j^2}$$
$$=\left(\sum_{r,s\ge0}-\sum_{r,s<0}\right)(-1)^{r+s}q^{Q_k(r,s)}+\left(\sum_{r,s\ge0}-\sum_{r,s<0}\right)(-1)^{r+s}q^{Q_{k}(r,s)+l_k(r,s)},$$
where we have computed $Q_k(r,s)=kr^2+(k+2)2rs+ks^2+k(r+s),$ and $l_k(r,s)=(2k+2)(r+s)+(2k+1),$ from making the substitutions $n=r+s$ and $j=r-s.$
Now this is equal to
$$f_{2k,2(k+2),2k}(q^{2k},q^{2k},q)+q^{2k+1}f_{2k,2(k+2),2k}(q^{4k+2},q^{4k+2},q)$$
$$=f_{k,k+2,k}(q^{2k},q^{2k},q^2)+q^{2k+1}f_{k,k+2,k}(q^{4k+2},q^{4k+2},q^2).$$
The last part of the theorem follows from noting which modular form needs to be multiplied by (2.17) to obtain a mock modular form, using known identities in the literature and that have been discussed herein.
\end{proof}
It is reasonable to conjecture that for every natural $k$ greater than $1,$ Theorem 2.3 gives rise to a mixed mock modular form. In Mortensen's and Hickerson's paper, it is noted about their [10, Theorem 0.9] that in the $k=1$ case, we may obtain three second order and eight eighth order mock theta functions (see [7] for eighth order functions). Another observation and application of our Theorem 2.1 follows from [15, (D1)] and is worth mentioning here. Namely, 
$$2\sum_{n\ge0}\sum_{n\ge j\ge 0}\frac{(q)_n(-q;q^2)_{j+1}(-1)^{n+j}q^{n(n+1)/2+n-j}}{(q^2;q^2)_{n-j}(-q^2;q^2)_{j}(1-q^{4j+2})}=\sum_{n\ge0}\frac{q^{n(n+1)/2}}{(-q)_n}$$
$$+\sum_{n\ge0}q^{n(3n+1)/2}(1+q^{2n+1})\sum_{|j|\le n}q^{-j^2}.$$
See reference [5] for an introduction to the distinct rank parity function, its connection with $\mathbb{Q}(\sqrt{6}),$ and similar functions related to real quadratic fields.

1390 Bumps River Rd. \\*
Centerville, MA
02632 \\*
USA \\*
E-mail: alexpatk@hotmail.com

\begin{thebibliography}{9}

\bibitem{ConcreteMath}
G.E. Andrews, \emph{The fifth and seventh order mock theta functions,} Trans. Amer. Math. Soc. 293 (1986), no. 1, 113--134.
\bibitem{ConcreteMath}
G.E. Andrews and D. Hickerson, \emph{Ramanujan's ``lost" notebook. VII. The sixth order mock theta functions,} Adv. Math. 89 (1991), no. 1, 60--105.

\bibitem{ConcreteMath}
G.E. Andrews, \emph{q-Orthogonal polynomials, Rogers-Ramanujan identities, and mock theta functions,} Proceedings of the Steklov Institute Dedicated to the 75th Birthday of A. A. Karatsuba, 276:21-32 (2012)

\bibitem{ConcreteMath}
D.M. Bressoud, M. Ismail, D. Stanton, \emph{Change of base in Bailey pairs,} Ramanujan J. 4 (2000), no. 4, 435--453.

\bibitem{ConcreteMath}
K. Bringmann and B. Kane, \emph{Multiplicative q-hypergeometric series arising from real quadratic fields,} Trans. Amer. Math. Soc. 363 (2011), no. 4, 2191--2209.

\bibitem{ConcreteMath}
Y.-S. Choi, \emph{Tenth order mock theta functions in Ramanujan's Lost Notebook,} Invent. Math. 136 (1999), 497--569.


\bibitem{ConcreteMath}
 B. Gordon and R.J. McIntosh, \emph{Some eighth order mock theta functions,} J. London Math. Soc. (2) 62 (2000),
321--335.
\bibitem{ConcreteMath}
E. Hecke, \emph{Uber einen Zusammenhang zwischen elliptischen Modulfunktionen und indefiniten quadra-
tischen Formen,} Mathematische Werke, Vandenhoeck and Ruprecht, Gottingen, 1959, pp. 418--427.
\bibitem{ConcreteMath} V.G. Kac and D.H. Peterson, \emph{Affine Lie algebras and Hecke modular forms,} Bull. Amer. Math. Soc. (N.S.) 3 (1980), 1057--1061.

\bibitem{ConcreteMath}
 D. Hickerson, E. Mortenson, \emph{Hecke-type double sums, Appell--Lerch sums, and mock theta functions I,} Proc.
London Math. Soc. 109 (2014), 382--422.


\bibitem{ConcreteMath}
J. Lovejoy, R. Osburn, \emph{The Bailey chain and mock theta functions,}
Adv. Math. 238 (2013), 442-458.

\bibitem{ConcreteMath}
J. Lovejoy, \emph{Bailey pairs and indefinite quadratic forms,} J. Math. Anal. Appl. 410 (2014), 1002-1013.

\bibitem{ConcreteMath}
J. Lovejoy, and R. Osburn, \emph{Mock theta double sums,}
Glasgow Math. J., to appear.


\bibitem{ConcreteMath}
A. Patkowski, \emph{On some new Bailey pairs and new expansions for some Mock theta functions,}
Methods and Applications of Analysis, Vol. 23, No. 2, pp. 205--214, 2016

\bibitem{ConcreteMath} D. Stanton, \emph{The Bailey-Rogers-Ramanujan group, in: q-Series with Applications to Combinatorics,} Number
Theory, and Physics, Contemporary Mathematics 291 (2001), 55--70.

\bibitem{ConcreteMath}
S. O. Warnaar, \emph{50 years of Bailey's lemma,} Algebraic combinatorics and applications (G$\ddot{o}\ss$weinstein, 1999),
333--347, Springer, Berlin, 2001.
\bibitem{ConcreteMath}
S. Zwegers, \emph{Mock Theta Functions,} PhD Thesis, Universiteit Utrecht (2002).


\end{thebibliography}
\end{document}